\documentclass[12pt]{amsart}

\usepackage[utf8]{inputenc}
\usepackage[english]{babel}

\usepackage[letterpaper,top=2cm,bottom=2cm,left=3cm,right=3cm,marginparwidth=1.75cm]{geometry}

\usepackage{amsmath}
\usepackage{graphicx}
\usepackage[colorlinks=true, allcolors=blue]{hyperref}

\usepackage{amssymb} 
\usepackage{comment}
\usepackage{indentfirst}
\usepackage{csquotes}
\usepackage{theoremref}
\usepackage{hyperref}
\usepackage{geometry}

\hypersetup{hidelinks, colorlinks=true, allcolors=cyan, pdfstartview=Fit, breaklinks=true}

\usepackage{amsthm}

\theoremstyle{plain}
\newtheorem{theorem}{Theorem}
\numberwithin{theorem}{section}

\newtheorem*{corollary*}{Corollary}
\newtheorem*{Example*}{Example}

\theoremstyle{definition}
\newtheorem{definition}{Definition}
\numberwithin{definition}{section}
\newtheorem*{def*}{Definition}
\newtheorem*{theorem*}{Theorem}

\newtheorem*{definition*}{Definition}

\theoremstyle{remark}

\newcommand{\bracket}[1]{\left( #1 \right)}
\newcommand{\floor}[1]{\left\lfloor #1 \right\rfloor}
\newcommand{\modulo}[3]{#1\equiv#2\ \bracket{\mathrm{mod}\ #3}}
\newcommand{\calB}{\mathcal{B}}
\newcommand{\calD}{\mathcal{D}}
\newcommand{\calP}{\mathcal{P}}
\newcommand{\calQ}{\mathcal{Q}}
\newcommand{\calS}{\mathcal{S}}
\newcommand{\bbZ}{\mathbb{Z}}
\newcommand{\bigdivide}[1]{\left.#1\right/}

\numberwithin{equation}{section}

\title{\textbf{There are infinitely many ($\MakeLowercase{\textbf{-1}}$,1)-Carmichael numbers}}
\author{QI-YANG ZHENG}
\date{} 

\address{Department of Mathematics, Sun Yat-sen University(Zhuhai Campus), Zhuhai}
\email{zhengqy29@mail2.sysu.edu.cn}

\begin{document}
\maketitle

\begin{abstract}
We prove that there exist infinitely many $(-1,1)$–Carmichael numbers, that is, squarefree, composite integers $n$ such that $p+1\,|\,n-1$ for each prime $p$ dividing $n$.
\end{abstract}


\section{Introduction}

The well-known Fermat's Little Theorem states that if $p$ is a prime, $p$ divides $a^p-a$ for each integers $a$. Carmichael numbers are composite numbers which have this property, i.e. a positive composite integer $n$ is a Carmichael number if $n$ divides $a^n-a$ for each integers $a$. In 1899, Korselt \cite{korselt1899probleme} give a criterion for Carmichael numbers: $n$ is a Carmichael number if and only if $n$ is squareefree and $p-1\,|\,n-1$ for each prime $p$ dividing $n$.

Now we define a series of numbers satisfying some conditions analogous to Korselt's criterion.

\begin{definition}
\textit{A positive squarefree composite number $n$ is an $(a,b)$-Carmichael numbers if $n\neq b$ and $p-a\,|\,n-b$ for each prime $p$ dividing $n$.}
\end{definition}

The condition of squarefree eliminates some trivial cases. For example, cube of prime numbers are $(-1,-1)$-Carmichael numbers, since $p^3+1=(p+1)(p^2-p+1)$.

Under the definition, $(1,1)$-Carmichael numbers are the usual Carmichael numbers; $(-1,-1)$-Carmichael numbers are known as Lucas-Carmichael numbers; $(k,k)$-Carmichael numbers are known as $k$-Korselt numbers.

These numbers play an important role in various primality tests. Thus we are interested in the properties of these numbers. Alford, Granville and Pomerance \cite{alford1994there} proved that there are infinitely many Carmichael numbers and stated that

~

\textit{One can modify our proof to show that for any fixed non-zero integer $a$, there are infinitely many squarefree, composite integers $n$ such that $p-a$ divides $n-1$ for all primes $p$ dividing $n$.}

~

In our notation it is equivalent to prove that there are infinitely many $(a,1)$-Carmichael numbers with $a$ is non-zero. However, there are no any result about infinitude of $(a,1)$-Carmichael numbers but $a=1$. In 2018, Wright \cite{wright2018there} proved that there are infinitely many $(-1,-1)$-Carmichael numbers, or Lucas-Carmichael numbers.

In this paper, we prove that there are infinitely many $(-1,1)$-Carmichael numbers, that is, squarefree, composite integers $n$ such that $p-a\,|\,n-1$ for each prime $p$ dividing $n$. The smallest one of $(-1,1)$-Carmichael numbers is $385=5\times7\times11$. One can see more examples in OEIS/A225711 \cite{sloane2007line}. The method we used is a modification of \cite{alford1994there}.

\begin{definition}
\textit{If $a$, $b$ are integers, we define $C(x;a,b)$ as the numbers of $(a,b)$-Carmichael numbers not exceeding $x$.}
\end{definition}

We will use $C(x)$ in the place of $C(x;-1,1)$ in this paper for convenience. In particular, we prove that $C(x)>x^{0.29}$ for all sufficiently large values of $x$. We believe that there is some gap between $x^{0.29}$ and the true size of $C(x)$, but it is enough to prove the infinitude of $(-1,1)$-Carmichael numbers.

Let $\pi(x)$ be the numbers of primes $p\leq x$. For a fixed non-zero integer $a$, let $\pi_a(x,y)$ be the number of primes $a<p\leq x$ for which $p-a$ is free of prime factors exceeding $y$. From \cite{friedlander1989shifted} we have

\begin{equation}
    \label{shifted primes without large factors}
    \pi_a(x,x^{1-E})\geq\gamma(E)\frac{x}{\log x}
\end{equation}

~

\noindent
for any $x\geq x_1$ and all non-zero integer $a$; $E\in(0,1-(2\sqrt e)^{-1})$; $\gamma(E)$ is a constant depends on $E$; $x_1$ depends on $a$ and $E$.

Define $\pi(x;d,a)$ to be the number of primes up to $x$ that belong to the arithmetic progression $a$ mod $d$. Let $\calB$ denote the set of numbers $B$ in the range $0<B<1$ for which there is a number $x_2(B)$ and a positive integer $D_B$ such that if $x\geq x_2(B)$, $(a,d)=1$ and $1\leq d\leq\min\{x^B,y/x^{1-B}\}$ then

\begin{equation}
    \label{primes in arithmetic progression}
    \pi(y;d,a)\geq\frac{\pi(y)}{2\varphi(d)}
\end{equation}

~

\noindent
whenever $d$ is not divisible by any member of $\calD_B(x)$, a set of at most $D_B$ integers, each of which exceeds $\log x$. We have $(0,5/12)\subset\calB$ (see \cite{alford1994there}). The main theorem of this paper depends intimately on the set $\calB$.

\begin{theorem}
\label{main theorem}
For each $B\in\calB$ and $E\in(0,1-(2\sqrt e)^{-1})$, there is a number $x_0=x_0(E,B)$ such that $C(x)\geq x^{EB}$ for all $x\geq x_0$.
\end{theorem}

Since $(0,5/12)\subset\calB$, we have $C(x)>x^{\beta-\varepsilon}$ for any $\varepsilon>0$ and all sufficiently large $x$, where

\begin{equation}
    \notag
    \beta=\bracket{1-\frac{1}{2\sqrt e}}\times\frac{5}{12}=0.290306\cdots
\end{equation}

~

\noindent
As mentioned above, we have $C(x)>x^{0.29}$ for all sufficiently large values of $x$.

~

\section{Subsequence products representing the identity in a group}

For a finite group $G$, let $n(G)$ denote the length of the longest sequence of (not necessarily distinct) elements of $G$ for which no non-empty subsequence has product the identity. Baker and Schmidt \cite{baker1980diophantine} gave an upper bound for $n(G)$ for arbitrary finite abelian groups, which is

\begin{theorem}
\label{n(G)}
If $G$ is a finite abelian group and $m$ is the maximal order of an element in $G$, then $n(G)<m(1+\log(|G|/m))$.
\end{theorem}

The idea is to construct an integer $L$ for which there are a very large number of primes $p$ such that $p+1$ divides $L$. Suppose that the product of some of these primes, say $C=p_1\cdots p_k$, is congruent to $1$ mod $L$, then $C$ is a $(-1,1)$-Carmichael number by definition. If we view these primes $p$ as elements of the group $G=(\bbZ/L\bbZ)^*$, then the condition becomes that, $C$ equals the identity of $G$. Though different primes $p$ may map to the same element of $G$, we view them as different elements of $G$, say label them with different numbers. The next result allows us to construct many such products \cite{alford1994there}.

\begin{theorem}
\label{many subsequences}
Let $G$ be a finite abelian group and let $r>t>n=n(G)$ be integers. Then any sequence of $r$ elements of $G$ contains at least $\left.\binom{r}{t}\right/\binom{r}{n}$ distinct subsequences of length at most $t$ and at least $t-n$, whose product is the identity.
\end{theorem}

Let $R$ be a sequence of $r$ elements of $G$. Though $R$ may contain same elements of $G$, we label these same elements with different numbers and view them as different elements. The step is necessary since the original proof may construct same subsequences, but we view them as different subsequences if their elements has different labeled numbers.

For example, let $G=(\bbZ/5\bbZ)^*$. We view $\{19,29\}$ and $\{19,59\}$ as different subsequences of $G$, although they are $\{4,4\}$ in the reduced residue class mod $5$.

~

\section{Infinitude of (-1,1)-Carmichael numbers}

First we prove a theorem resembling Theorem 3.1 of \cite{alford1994there} and the proof is based on that.

\begin{theorem}
\label{Prachar}
Suppose that $B\in\calB$ and $a$ is a fixed non-zero integer. There exists a number $x_3(B)$ such that if $x\geq x_3(B)$ and $L$ is a squarefree integer not divisible by any prime exceeding $x^{(1-B)/2}$ and for which $\sum_{\mathrm{prime}\ q|L}1/q\leq(1-B)/32$, then there is a positive integer $k\leq x^{1-B}+|a|$ with $(k,aL)=1$, such that

\begin{equation}
    \notag
    \#\{ d|L:dk+a\leq x,\ dk+a\text{ is prime} \}\geq\frac{2^{-D_B-\omega(a)-3}}{\log x}\#\{ d|L:1\leq d\leq x^B \},
\end{equation}

~

\noindent
where $\omega(n)$ is the number of distinct prime factors of $|n|$.
\end{theorem}

\begin{proof}
Let $x_3(B)=\max\{x_2(B),17^{(1-B)^{-1}},|a|^{(1-B)^{-1}}\}$. For each $d\in\calD_B(x)$ which divides $L$, we divide some prime factor of $d$ out from $L$. Furthermore, we divide all prime factors of $(a,L)$ out from $L$, so as to obtain a number $L'$ which is not divisible by any number in $\calD_B(x)$ and $(a,L')=1$. Thus $\omega(L')\geq\omega(L)-D_B-\omega(a)$, and

\begin{equation}
    \label{L' and L}
    \#\{ d|L':1\leq d\leq y \}\geq2^{-D_B-\omega(a)}\#\{ d|L:1\leq d\leq y \}
\end{equation}

~

\noindent
for any $y\geq1$. To see this, think of a divisor $d'$ of $L'$ as corresponding to a divisor $d$ of $L$ if and only if $d'$ divides $d$ and $d/d'$ divides $L/L'$. So if $d\leq y$ then the corresponding $d'$ is at most $y$. Moreover, for any divisor $d'$ of $L'$, the number of divisors $d$ of $L$ which correspond to $d'$ is at most the number of divisors of $L/L'$, which is at most $2^{D_B+\omega(a)}$.

From \eqref{primes in arithmetic progression} we see that, for each divisor $d$ of $L'$ with $1\leq d\leq x^B$, we have

\begin{equation}
    \label{lower bound}
    \pi(dx^{1-B};d,a)\geq\frac{\pi(dx^{1-B})}{2\varphi(d)}\geq\frac{dx^{1-B}}{2\varphi(d)\log(dx^{1-B})}\geq\frac{dx^{1-B}}{2\varphi(d)\log x},
\end{equation}

~

\noindent
since $\pi(y)\geq y/\log y$ for all $y\geq17$ (see \cite{rosser1962approximate}). Furthermore, since any prime factor $q$ of $L$ is at most $x^{(1-B)/2}$, we can use Montgomery and Vaughan's explicit version of the Brun-Titchmarsh theorem \cite{montgomery1973large}, to get

\begin{equation}
    \notag
    \pi(dx^{1-B};dq,a)\leq\frac{2dx^{1-B}}{\varphi(dq)\log(x^{1-B}/q)}\leq\frac{4}{\varphi(q)(1-B)}\frac{dx^{1-B}}{\varphi(d)\log x}\leq\frac{8}{q(1-B)}\frac{dx^{1-B}}{\varphi(d)\log x}.
\end{equation}

~

Therefore, by \eqref{lower bound}, the number of primes $p\leq dx^{1-B}$ with $\modulo{p}{a}{d}$ and $((p-a)/d,L)=1$ is at least

\begin{equation}
    \notag
    \begin{aligned}
    &\ \ \ \ \pi(dx^{1-B};d,a)-\sum_{\text{prime }q|L}\pi(dx^{1-B};dq,a)\\
    &\geq\bracket{\frac{1}{2}-\frac{8}{1-B}\sum_{\text{prime }q|L}\frac{1}{q}}\frac{dx^{1-B}}{\varphi(d)\log x}\\
    &\geq\frac{x^{1-B}}{4\log x}.
    \end{aligned}
\end{equation}

~

\noindent
Thus we have at least

\begin{equation}
    \notag
    \frac{x^{1-B}}{4\log x}\#\{d|L':1\leq d\leq x^B\}
\end{equation}

~

\noindent
pairs $(p,d)$ where $p\leq dx^{1-B}$ is a prime, $\modulo{p}{a}{d}$, $((p-a)/d,L)=1$, $d|L'$ and $1\leq d\leq x^B$. Each such pair $(p,d)$ corresponds to an integer $(p-a)/d$ that is coprime to $L$ and

\begin{equation}
    \notag
    \frac{p-a}{d}\leq x^{1-B}+\frac{|a|}{d}\leq x^{1-B}+|a|.
\end{equation}

~

\noindent
Since $x^{1-B}\geq|a|$, there is at least one integer $k\leq x^{1-B}+|a|$ with $(k,L)=1$ such that $k$ has at least

\begin{equation}
    \notag
    \frac{x^{1-B}}{x^{1-B}+|a|}\frac{1}{4\log x}\#\{d|L':1\leq d\leq x^B\}\geq\frac{1}{8\log x}\#\{d|L':1\leq d\leq x^B\}
\end{equation}

~

\noindent
representations as $(p-a)/d$ with $(p,d)$ as above. Moreover, we have $(a,d)=1$ since $(a,L')=1$, so

\begin{equation}
    \notag
    \bracket{\frac{p-a}{d},a}=(p-a,a)=(p,a).
\end{equation}

~

\noindent
If $(p,a)=p$, then $a=p$ since $p\geq a$, which is a contradiction to the condition $((p-a)/d,L)=1$. Thus for this integer $k$ we have $(k,a)=1$ and

\begin{equation}
    \notag
    \begin{aligned}
    &\ \ \ \ \#\{ d|L:dk+a\leq x,\ dk+a\text{ is prime} \}\\
    &\geq\frac{1}{8\log x}\#\{d|L':1\leq d\leq x^B\}\\
    &\geq\frac{2^{-D_B-\omega(a)-3}}{\log x}\#\{ d|L:1\leq d\leq x^B \},
    \end{aligned}
\end{equation}

~

\noindent
where we use \eqref{L' and L} in the last inequality.

\end{proof}

\noindent
Now we recalling the main theorem of this paper.

~

\noindent
\textbf{Theorem 1.1.} \textit{For each $B\in\calB$ and $E\in(0,1-(2\sqrt e)^{-1})$, there is a number $x_0=x_0(E,B)$ such that $C(x)\geq x^{EB}$ for all $x\geq x_0$.}

\begin{proof}

Let $0<\varepsilon<EB$ be a fixed number, $\theta=(1-E)^{-1}$ and let $y\geq2$ be a parameter. Denote by $\calQ$ the set of primes $q\in(y^\theta/\log y,y^\theta]$ for which $q-1$ is free of prime factors exceeding $y$. By \eqref{shifted primes without large factors},

\begin{equation}
    \label{calQ}
    \begin{aligned}
    |\calQ|&=\pi_1(y^\theta,y)-\pi_1(y^\theta/\log y,y)\\
    &\geq\gamma(E)\frac{y^\theta}{\log y^\theta}-\frac{2y^\theta}{(\log y)\log(y^\theta/\log y)}\\
    &\geq\frac{\gamma(E)}{2}\frac{y^\theta}{\log y^\theta}
    \end{aligned}
\end{equation}

~

\noindent
for all sufficiently large $y$. Let $L$ be the product of the primes $q\in\calQ$, then

\begin{equation}
    \label{log(L)}
    \log L\leq|\calQ|\log y^\theta\leq\pi(y^\theta)\log y^\theta\leq2y^\theta
\end{equation}

~

\noindent
for all sufficiently large $y$. Let $\lambda(n)$ denote the Carmichael lambda function, the largest order of an element in $(\bbZ/n\bbZ)^*$. We have $\lambda(L)$ is the least common multiple of the numbers $q-1$ for those primes $q$ that divide $L$. Since each such $q-1$ is free of prime factors exceeding $y$, we know that if the prime power $p^s$ divides $\lambda(L)$ then $p\leq y$ and $p^s\leq y^\theta$. Thus if we let $p^{s_p}$ be the largest power of $p$ with $p^{s_p}\leq y^\theta$, then

\begin{equation}
    \label{lambda(L)}
    \lambda(L)\leq\prod_{p\leq y}p^{a_p}\leq\prod_{p\leq y}y^\theta=y^{\theta\pi(y)}\leq e^{2\theta y}
\end{equation}

~

\noindent
for all sufficiently large $y$. Let $\delta=\varepsilon\theta/(4B)$ and let $x=e^{y^{1+\delta}}$. By Theorem 5 of \cite{rosser1962approximate} we have

\begin{equation}
    \notag
    \begin{aligned}
    \sum_{\text{prime }q|L}\frac{1}{q}&\leq\sum_{y^\theta/\log y<q\leq y^\theta}\frac{1}{q}\\
    &\leq\log\log y^\theta+\frac{1}{2(\log y^\theta)^2}-\log\log(y^\theta/\log y)+\frac{1}{2(\log(y^\theta/\log y))^2}\\
    &=\log\frac{\theta\log y}{\theta\log y-\log\log y}+\frac{1}{2(\log y^\theta)^2}+\frac{1}{2(\log(y^\theta/\log y))^2}\\
    &\leq\frac{1-B}{32}
    \end{aligned}
\end{equation}

~

\noindent
for all sufficiently large $y$. Then we can apply Theorem 3.1 with $B$, $x$, $L$. Thus for all sufficiently large $y$, there is an integer $k$ coprime to $L$, for which the set $\calP$ of primes $p\leq x$ with $p=dk-1$ for some divisor $d$ of $L$, satisfies

\begin{equation}
    \label{calP}
    |\calP|\geq\frac{2^{-D_B-3}}{\log x}\#\{d|L:1\leq d\leq x^B\}.
\end{equation}

~

\noindent
Let $G=(\bbZ/L\bbZ)^*\times (\bbZ/2\bbZ)^+$. Since $\lambda(n)$ is even for all $n\in\bbZ_{\geq3}$, we conclude from Theorem \ref{n(G)}, \eqref{log(L)} and \eqref{lambda(L)} that

\begin{equation}
    \label{n(G) upper bound}
    n(G)<\lambda(L)\bracket{1+\log\frac{2\varphi(L)}{\lambda(L)}}\leq\lambda(L)(1+\log2+\log L)\leq e^{3\theta y}
\end{equation}

~

\noindent
for all sufficiently large $y$. The product of any

\begin{equation}
    \notag
    u:=\floor{\frac{\log x^B}{\log y^\theta}}=\floor{\frac{B\log x}{\theta\log y}}
\end{equation}

~

\noindent
distinct prime factors of $L$, is a divisor $d$ of $L$ with $d\leq x^B$. We deduce from \eqref{calQ} that

\begin{equation}
    \notag
    \#\{d|L:1\leq d\leq x^B\}\geq\binom{\omega(L)}{u}\geq\bracket{\frac{\omega(L)}{u}}^u\geq\bracket{\frac{\gamma(E)y^\theta}{2B\log x}}^u=\bracket{\frac{\gamma(E)}{2B}y^{\theta-1-\delta}}^u.
\end{equation}

~

\noindent
Thus, by $\eqref{calP}$ and the identity $(\theta-1-\delta)B/\theta=EB-\varepsilon/4$, we have

\begin{equation}
    \label{calP2}
    |\calP|\geq\frac{2^{-D_B-3}}{\log x}\bracket{\frac{\gamma(E)}{2B}y^{\theta-1-\delta}}^{\floor{\frac{B\log x}{\theta\log y}}}
\end{equation}

~

\noindent
for all sufficiently large $y$. Now take $\calP'=\calP\backslash\calQ$. Since $|\calQ|\leq y^\theta$, we have by \eqref{calP2} that

\begin{equation}
    \label{calP'}
    |\calP'|\geq x^{EB-\varepsilon/2}
\end{equation}

~

\noindent
for all sufficiently large $y$.

For each element of $p\in\calP'$, we view it as the element $(\overline{p},-1)\in G$, where $\overline{p}$ denotes the residue of $p$ in $(\bbZ/L\bbZ)^*$. As mentioned above, if $\overline{p_i}=\overline{p_j}$ but $p_i\neq p_j$, then we view $(\overline{p_i},-1)$ and $(\overline{p_j},-1)$ as two different elements of $G$. If $\calS$ is a subsequence of $G$ that contains more than one element and if

\begin{equation}
    \notag
    \Pi(\calS):=\prod_{g\in\calS}g=1_G,
\end{equation}

~

\noindent
then we can construct a $(-1,1)$-Carmichael number. Firstly, $|\calS|$ must be even, since the second component of all elements are $-1$. Moreover, the product of the first component of all elements is $1$ mod $L$, say

\begin{equation}
    \notag
    \Pi_1(\calS):=\prod_{p}\modulo{p}{1}{L},
\end{equation}

~

\noindent
where the product is over all preimages of first component of elements in $\calS$. Furthermore, each member of $\calP$ is $-1$ mod $k$ so that $\modulo{\Pi_1(\calS)}{1}{k}$, since $|\calS|$ is even. Thus $\modulo{\Pi_1(\calS)}{1}{kL}$ since $(k,L)=1$. If $p\in\calP'$ then $p\in\calP$ so that $p+1=kd\,|\,kL\,|\,\Pi_1(\calS)-1$. Thus $\Pi_1(\calS)$ is a $(-1,1)$-Carmichael number.

Let $t=e^{y^{1+\delta/2}}$. Evidently $t>n(G)$ for all sufficiently large $y$. Then by Theorem \ref{many subsequences}, we see that the number of $(-1,1)$-Carmichael numbers of the form $\Pi_1(\calS)$, where $\floor{t}-n(G)\leq|\calS|\leq\floor{t}$, is at least

\begin{equation}
    \notag
    \bigdivide{\binom{|\calP'|}{\floor{t}}}\binom{|\calP'|}{n(G)}\geq\bigdivide{{\bracket{\frac{|\calP'|}{\floor{t}}}^{\floor{t}}}}|\calP'|^{n(G)}\geq\bracket{x^{EB-\varepsilon/2}}^{\floor{t}-n(G)}\floor{t}^{-\floor{t}}\geq x^{t(EB-\varepsilon)}
\end{equation}

~

\noindent
for all sufficiently large $y$, using $\eqref{n(G) upper bound}$ and $\eqref{calP'}$. Since each $(-1,1)$-Carmichael number we construct satisfy $\Pi_1(\calS)\leq x^t$, we have $C(X)\geq X^{EB-\varepsilon}$ for all sufficiently large $y$, where $X:=x^t$. Moreover, since $y$ and $X$ can be uniquely determined by each other, we derive $C(x)\geq x^{EB-\varepsilon}$ for each $B\in\calB$, $E\in(0,1-(2\sqrt e)^{-1})$, $0<\varepsilon<EB$ and sufficiently large $x$.

Since $E$ is in an open interval, choose $E'\in(0,1-(2\sqrt e)^{-1})$ with $E'>E$ and let $\varepsilon=(E'-E)B$, then we have $C(x)\geq x^{E'B-(E'-E)B}=x^{EB}$ for sufficiently large $x$, this completes the proof.

\end{proof}

Unfortunately, our proof is not applicable to the case $a\neq1$. In fact, we can only solve the case that $a$ have a small order mod $k$ (see Theorem \ref{Prachar}). Since the properties of $k$ is unknown, $k$ may ruin the estimate of some arguments. Anyway, we believe that for non-zero integers $a$, $b$, there are infinitely many $(a,b)$-Carmichael numbers.

\section*{Acknowledgement}

The ideas came to us after seeing the papers \cite{alford1994there} and \cite{wright2018there}. The method we used in this paper is a simple modification of method in \cite{alford1994there}.

~


\end{document}